\documentclass[10pt,leqno]{amsart}
\usepackage{indentfirst, amssymb,amsmath,amsthm}
\newtheorem*{theoA}{Theorem A}
\newtheorem*{theoB}{Theorem B}
\newtheorem*{theoC}{Theorem C}
\newtheorem*{theoD}{Theorem D}
\newtheorem*{theoE}{Theorem E}
\newtheorem*{theoF}{Theorem F}

\newtheorem{theo}{Theorem}[section]
\newtheorem{lem}{Lemma}[section]
\newtheorem{cor}{Corollary}[section]

\newtheorem{exm}{Example}[section]
\newtheorem{defi}{Definition}[section]
\newtheorem{ques}{Question}[section]
\newtheorem{rem}{Remark}[section]
\newcommand{\ol}{\overline}
\newcommand{\be}{\begin{equation}}
\newcommand{\ee}{\end{equation}}
\newcommand{\beas}{\begin{eqnarray*}}
\newcommand{\eeas}{\end{eqnarray*}}
\newcommand{\bea}{\begin{eqnarray}}
\newcommand{\eea}{\end{eqnarray}}
\newcommand{\lra}{\longrightarrow}
\numberwithin {equation}{section}
\numberwithin {lem}{section}
\numberwithin {theo}{section}
\numberwithin {defi}{section}
\numberwithin {rem}{section}
\numberwithin {cor}{section}
\renewcommand{\vline}{\mid}
\begin{document}
	\title[Uniqueness and two shared set problems of $\boldmath{L}$-Function ]{ Uniqueness and two shared set problems of $\boldmath{L}$-Function and certain class of meromorphic function}
		\date{}
	\author{Abhijit Banerjee\;\;\; and \;\;\;Arpita Kundu}
	\date{}
	\address{ Department of Mathematics, University of Kalyani, West Bengal 741235, India.}
	\email{abanerjee\_kal@yahoo.co.in, abanerjeekal@gmail.com}
	
	\address{Department of Mathematics, University of Kalyani, West Bengal 741235, India.}
	\email{arpitakundu.math.ku@gmail.com}
	\maketitle
	\let\thefootnote\relax
	\footnotetext{Type set by \AmS -\LaTeX}
	\begin{abstract}
	Starting with a question of Yuan-Li-Yi [Value distribution of $L$-functions and uniqueness questions of F. Gross, Lithuanian Math. J., $\mathbf{58(2)}$(2018), 249-262] we have studied the uniqueness of a meromorphic function $f$ and an $L$-function $\mathcal{L}$ sharing two finite sets. At the time of execution of our work, we have pointed out a serious lacuna in the proof of a recent result of a of Sahoo-Halder [ Some results on L-functions related to sharing two finite sets, Comput. Methods Funct. Theo., $\mathbf {19}$(2019), 601-612] which makes most of the part of the Sahoo-Halder's paper under question. In context of our choice of sets, we have rectified Sahoo-Halder's result in a convenient manner.     
	\end{abstract}
\section{introduction}	
By a meromorphic function we shall always mean a meromorphic function in the
complex plane. We adopt the standard notations of Nevanilinna theory of meromorphic functions as explained in \cite{W.K.Hayman_64}.
Let   $\mathbb{\overline{C}}=\mathbb{C}\cup\{\infty\}$, $\mathbb{C^*}=\mathbb{C}\setminus\{0\}$ and $ \mathbb{\ol N}=\mathbb{N}\cup \{0 \}$, where $\mathbb{C}$ and $\mathbb{N}$  denote the set of all complex numbers and natural numbers respectively and by $\mathbb{Z}$ we denote the set of all integers. 
For any non-constant meromorphic function $h(z)$ we define 
$S(r,h)=o(T(r,h)), ( r\lra \infty, r\not\in E)$ where $E$ denotes any set of positive real numbers having finite linear measure.
\par \begin{defi}
	Let for a  non-constant meromorphic function $f$ and  $S\subset\overline{\mathbb{C}}$,  $E_{f}(S)=\bigcup_{a\in S}\{(z,p)\\\in\mathbb{C}\times\mathbb{N}: f(z)=a\; with\; multiplicity\; p\}$ $\left(\ol  E_{f}(S)=\bigcup_{a\in S}\{(z,1)\in\mathbb{C}\times\mathbb{N}: f(z)=a\}\right) $. Then we say  $f$, $g$ share the set $S$ Counting Multiplicities (CM)(Ignoring Multiplicities (IM)) if $E_{f}(S)=E_{g}(S)$ $\left(\ol  E_{f}(S)=\ol E_{g}(S)\right) $. 
\end{defi} 
When $S$ contains only one element the definition coincides with the classical definition of value sharing.

 This paper deals with the uniqueness problems of set sharing related to $L$-functions and an arbitrary meromorphic function in $\mathbb{C}$.
 In 1989, Selberg \cite{selberg-92} found new class of Dirichlet series, called as Selberg class, which in course of time made a significant impact on the realm of research in analytic number theory. Throughout  this paper an $L$-function means actually a Selberg class function with the Riemann zeta function
 as the prototype. The
Selberg class $\mathcal{S}$ of $L$-functions is the set of all Dirichlet series $\mathcal{L}(s)=\sum_{n=1}^{\infty}
a(n)n^{-s}$ of a complex variable
$s$ that satisfy the following axioms (see \cite{selberg-92}):
\\$\boldsymbol{(i)}$ Ramanujan hypothesis: $a(n) \ll n^{\epsilon}$ for every $\epsilon > 0$.
\\$\boldsymbol{(ii)}$ Analytic continuation: There is a nonnegative integer $k$ such that $(s- 1)^{k}\mathcal{L}(s)$ is an entire function of finite order.
\\$\boldsymbol{(iii)}$ Functional equation: $\mathcal{L}$ satisfies a functional equation of type $$\Lambda _{\mathcal{L}}(s) = \omega\ol {\Lambda _{\mathcal{L}}(1 - \ol s)},$$ where $$\Lambda _{\mathcal{L}}(s)=\mathcal{L}(s)Q^{s}\prod_{j=1}^{K}\Gamma(\lambda _{j}s + \nu_{j})$$ with positive real numbers Q, $\lambda_{j}$ and complex numbers $\nu_{j} , \omega$  with 
$ Re \nu_{j}\geq 0$ and $|\omega|=1$.
\\$\boldsymbol{(iv)}$ Euler product hypothesis : $\mathcal{L}$ can be written over prime as $$\mathcal{L}(s)=\prod_{p}\exp\left(\sum_{k=1}^{\infty}b(p^{k})/p^{ks}\right)$$ with suitable coefficients $b(p^{k})$ satisfying $b(p^{k})\ll p^{k\theta}$ for some $\theta<1/2$ where the product is taken over all prime numbers $p$.
\par The Ramanujan hypothesis implies that the Dirichlet series $\mathcal{L}$ converges absolutely in the half-plane $Re( s) > 1$ and then is extended meromorphically. The degree $d_{\mathcal{L}}$ of an $L$-function $\mathcal{L}$ is defined to be \beas d_{\mathcal{L}}=2\sum_{j=1}^{K}\lambda_{j},\eeas

\par where $\lambda_{j}$ and $K$ respectively be the positive real number and the positive integer as in axiom (iii) above.
For the last few years, the researchers  have found an increasing interest on the value distributions of $L$-functions. Readers can make a glance over the references (\cite{Graun-Grahl-Steuding}, \cite{B.Q.LI-Proc.Am-10}, \cite{Li-Yi_Nachr}, \cite{Steuding-Sprin-07}). Like meromorphic function, the value distribution of an $L$-function $\mathcal{L}$ is actually the scattering of the roots of the equation $\mathcal{L}(s)=c$ for some $c\in \mathbb{C}\cup\{\infty\}$. By the sharing of sets by an $L$-function, we mean the same notion as mentioned in the first and second paragraph of this paper where all the definitions discussed also applicable to an $L$-function. In 2007, Steuding [p. 152, \cite{Steuding-Sprin-07} ] first studied the uniqueness problem of two $\mathcal{L}$ functions and obtained a remarkable result in connection to Nevanlinna $5$ point uniqueness theorem for meromorphic function. In \cite{Steuding-Sprin-07} it was shown that only one share value is enough to determine an $\mathcal{L}$ function under certain hypothesis. The result was as follows:
\begin{theoA}\cite{Steuding-Sprin-07}
	If two $L$-functions $\mathcal{L}_{1}$ and $\mathcal{L}_{2}$ with $a(1)=1$ share a complex value $c\;(\not=\infty)$ CM, then $\mathcal{L}_{1}=\mathcal{L}_{2}$.\end{theoA}

	Hu-Li \cite{Hu_LI_Can-16} found a counterexample to show that {\it{Theorem A}} is not true when $c = 1$. 

Since $L$-functions possess meromorphic continuations, researchers presumed that there might be an intimate relationship between $L$-function and arbitrary meromorphic function under sharing of values. In 2010, Li \cite{B.Q.LI-Proc.Am-10} exhibited the following example to show that {\it{Theorem A}} cease to hold for an $L$-function and a meromorphic function.
\begin{exm}
	For an entire function g, the functions $\zeta$ and $\zeta e^{g}$ share $0$ CM, but $\zeta \not=\zeta e^{g}$.
\end{exm}
However, corresponding to two distinct complex values, Li \cite{B.Q.LI-Proc.Am-10} was able to obtain the following uniqueness result.
\begin{theoB}\cite{B.Q.LI-Proc.Am-10}
	Let $f$ be a meromorphic function in $\mathbb{C}$ having finitely many poles and let a and b be any two distinct finite complex values. If $f$ and a non constant $L$-function $\mathcal{L}$ share $a$ CM and $b$ IM, then $f=\mathcal{L}$.
\end{theoB}  
	For three IM shared values, Li-Yi \cite{Li-Yi_Nachr} obtained the following
theorem.
\begin{theoC}\cite{Li-Yi_Nachr} Let f be a transcendental meromorphic function in C having finitely
	many poles in $\mathbb{C}$, and let $b_{1},\;b_{2},\; b_{3}$ be three distinct finite complex values. If $f$ and a non-constant $L$-function $\mathcal{L}$ shares
	$b_{1},\;b_{2},\; b_{3}$  IM,  then $\mathcal{L}\equiv f$ .
\end{theoC}
Inspired by the famous question of Gross \cite{G.F_Springer(1977)}, regarding uniqueness and sharing of sets, a lot of investigations were performed by many researchers. So the analogous question for the uniqueness of meromorphic function $f$ and an $L$-function $\mathcal{L}$ needs further attention. In this regard, Yuan-Li-Yi \cite{Yan-Li-Yi_Lith} proposed the following question: \begin{ques}
	What can be said about the relationship
	between a meromorphic function $f$ and an $L$-function $\mathcal{L}$ if $f$ and $\mathcal{L}$ share one or two finite sets?
\end{ques}   \par
In response to their own question Yuan-Li-Yi \cite{Yan-Li-Yi_Lith} proved the following uniqueness result.
\begin{theoD}\cite{Yan-Li-Yi_Lith}
	Let $S=\{a_{1},a_{2},\ldots,a_{l}\}$, where ${a_{1},a_{2},\ldots,a_{l}}$ are all distinct roots of the algebraic equation
	$w^{n} + aw^{m} + b = 0$. Here $l$ is a positive integer satisfying $1\leq l\leq n$, $n$ and $m$ are relatively prime positive integers with $n \geq 5$ and $n > m$, and $a$, $b$, $c$ are three nonzero finite constants, where $c \not= \alpha_{j}$ for $1 \leq j \leq l$. Let
	$f$ be a meromorphic function having finitely many poles in $\mathbb{C}$, and let ${\mathcal{L} } $ be a non constant $L$-function. If $f$ and $\mathcal{L}$ share $S$ CM and $c$ IM, then $f\equiv  \mathcal{L}$.
\end{theoD}
In the mean time, considering the sharing of two finite sets Lin-Lin \cite{Lin-Lin-filomat} proved the following theorem.
	\begin{theoE}\cite{Lin-Lin-filomat}
	Let $f$ be a meromorphic function in $\mathbb{C}$ with finitely many poles, $S_{1}, S_{2}\subset \mathbb{C}$ be two distinct sets such that $S_{1}\cap  S_{2}=\phi$ and $\#(S_{i} ) \leq 2$, $i = 1$, $2$, where $\#(S)$ denotes
	the cardinality of the set $S$. Suppose for a finite set $S = \{a_{i} \mid i = 1, 2,\ldots, n\}$, $C(S)$	is defined by $C(S) = \frac{1}{n}\sum_{i=1}^{n}a_{i}$. If $f$ and a non-constant $L$-function $\mathcal{L}$ share $S_{1}$ CM
	and $S_{2}$ IM, then (i) $\mathcal{L} = f$ when $C(S_{1}) \not= C(S_{2})$ and (ii) $\mathcal{L} = f$ or $\mathcal{L}+ f = 2C(S_{1})$	when $C(S_{1}) = C(S_{2})$.
\end{theoE}

In the same paper Lin-Lin \cite{Lin-Lin-filomat}, asked the following question:
 
  \begin{ques}
	What can be said about the conclusions of Theorem E if max $\{ \#(S_1), \#(S_2)\} \geq 3$?
\end{ques} To provide an answer to the question of Lin-Lin \cite{Lin-Lin-filomat}, Sahoo-Halder \cite{Hal-Sahoo-cmft} obtained the following result which is also pertinent to {\it Question 1.1}.  
\begin{theoF}\cite{Hal-Sahoo-cmft}
	Let $f$ be a meromorphic function in $\mathbb{C}$ with finitely many poles, and $m(\geq 3)$
	be a positive integer. Suppose that $ S_{1} = \{a_{1}, a_{2}, \ldots,a_{m}\}$, $S_{2} =\{b_{1}, b_{2}\}$ be two
	subsets of $\mathbb{C}$ such that $S_{1}\cap  S_{2}=\phi$ and $(b_{1} - a_{1})^{2}(b_{1}- a_{2})^{2}\ldots (b_{1} - a_{m})^{2} \not=
	(b_{2} -a_{1})^{2}(b_{2} - a_{2})^{2}\ldots (b_{2} - a_{m})^{2}$. If $f$ and a non-constant $L$-function $\mathcal{L}$ share $S_{1}$
	IM and $S_{2}$ CM, then $\mathcal{L} = f$.
\end{theoF}
The above theorem is one of the salient result in \cite{Hal-Sahoo-cmft} and the proof of the same contains the major portion of the paper. 
\begin{rem}
	In the proof of {\it{Theorem F}}, in \cite{Hal-Sahoo-cmft} [p. 608, before (3.3) in the proof of Theorem 1.2] the authors concluded that if $f$ and $\mathcal{L}$ share $S_{1} = \{a_{1}, a_{2}, \ldots,a_{m}\}$ IM and $S_2=\{b_{1},b_{2}\}$ CM, then $P(f)=(f-a_{1})(f-a_{2})\ldots(f-a_{m})$ and $P(\mathcal{L})=(\mathcal{L}-a_{1})(\mathcal{L}-a_{2})\ldots(\mathcal{L}-a_{m})$ share say $S_3=\{c_1,c_2\}$ CM, where $c_1=(b_{1}-a_{1})(b_{1}-a_{2})\ldots(b_{1}-a_{m})$ and $c_2=(b_{2}-a_{1})(b_{2}-a_{2})\ldots(b_{2}-a_{m})\}$ with $c_1^{2}\not =c_2^{2}$. With the help of this argument subsequently (see Case 2.1, (3.4) in the proof of Theorem 1.2 \cite{Hal-Sahoo-cmft}) they set up an entire function $V=e^{u}$ and obtained $T(r,e^{u}/H)=O(r)$ for some rational function $H$ and using this, they proved the rest part of the theorem.\par
	In general, from the basic definition of sharing of sets  this argument is not true for any arbitrary $f$ and $\mathcal{L}$. Below we are explaining the facts: \par
	We first note that given $f$ and $\mathcal{L}$ share the set $S_2=\{b_{1},b_{2}\}$ CM, so any $b_1$ $(b_{2})$ point of $f$ $(\mathcal{L})$ of order say $p$ will be a $b_i$ $(i=1,2)$ point of $\mathcal{L}$ $(f)$ of order $p$. Then noting the  definition of CM sharing of sets we know $P(f)$ and $P(\mathcal{L})$ will share the set $S_{3}$ CM only when the left hand side of the following equation 
	$$P^{2}(h)-(c_1+c_2)P(h)+c_1 c_2=0$$
	could be factorized in the form $(h-b_1)^{m}(h-b_2)^{m}$, where $h=f$ or $\mathcal{L}$ with $c_1^{2}\not =c_2^{2}$. But this is not always possible for any arbitrary choice of $a_i$'s, ($i=1,2,\ldots,m$) and $b_i$'s, ($i=1,2$). When $S_1$ contains one element say $a_1\not =\frac{b_1+b_2}{2}$ then $P^{2}(h)-(c_1+c_2)P(h)+c_1 c_2=(h-a_1)^{2}-(b_1+b_2-2a_1)(h-a_1)+(b_1-a_1)(b_2-a_1)=(h-b_1)(h-b_2)$. If $S_1$ contains two elements say $a_1$, $a_2$, then it is easy to verify $$P^{2}(h)-(c_1+c_2)P(h)+c_1 c_2=(h-b_1)(h-a_1-a_2+b_1)(h-b_2)(h-a_1-a_2+b_2)$$ and so in this case also the arguments in \cite{Hal-Sahoo-cmft} does not hold. Next when $m=3$ that is $S_1$ contains $3$ elements, say $a_1=i$, $a_2=-i$, $a_3=-1$, then considering $b_1=1$, $b_2=0$ it is easy to verify that $c_1=4$ and $c_2=1$. But
	$$P^{2}(h)-5P(h)+4\not=h^{3}(h-1)^{3}.$$ In fact, it is easy to verify that $1$ and $0$ are the simple roots of the equation
	$(t^{2}+1)^{2}(t+1)^{2}-5(t^{2}+1)(t+1)+4=0$. So one can observe a big gap in the proof of {\it Theorem 1.2} \cite {Hal-Sahoo-cmft} and the theorem cease to hold. Actually the condition $c_1^{2}\not =c_2^{2}$ is not sufficient enough to factorize the expression $P^{2}(t)-(c_1+c_2)P(t)+c_1 c_2$ into the form $(t-b_1)^{m}(t-b_2)^{m}$ except for the case $m=1$ with $a_1\not =\frac{b_1+b_2}{2}$. \par  As the entire analysis of Theorem 1.2 \cite {Hal-Sahoo-cmft} is depending upon the statement that if $f$ and $\mathcal{L}$ share $S_{1}$ IM and $S_2$ CM, then $P(f)$ and $P(\mathcal{L})$ share $S_3=\{c_1,c_2\}$ CM, Theorem 1.2 \cite {Hal-Sahoo-cmft} is not valid.   
\end{rem}
\par In this paper though our prime intention is to provide an answer to the question of Yuan-Li-Yi \cite{Yan-Li-Yi_Lith}, but at the same time we have somehow been able to present the corrected form of {\it Theorem F}  concerning a special set introduced in \cite{Ban-Mallick-bohemica} which in turn answer {\it Question 1.2}. 

\par We require the following definitions  for the main results of the paper.
\begin{defi}\cite{Lahiri-comp.var}
	Let k be a nonnegative integer or infinity. For $a\in\overline{\mathbb{C}}$  we
	denote by $E_k(a; f)$ the set of all a-points of f, where an $a$-point of multiplicity $m$ is counted
	$m$ times if $m\leq k$ and $k + 1$ times if $m > k$. If $E_k(a; f) = E_k(a; g)$, we say that $f$, $g$ share
	the value a with weight $k$.
\end{defi}
We write $f$, $g$ share $(a,k)$ to mean that $f$, $g$ share the value $a$ with weight $k$. Clearly if
$f$, $g$ share $(a,k)$ then $f$, $g$ share $(a,p)$ for any integer $p$, $0\leq p < k$. Also we note that $f$, $g$
share a value $a$ IM or CM if and only if $f$, $g$ share $(a,0)$ or $(a,\infty)$ respectively.
\begin{defi}\cite{Lahiri-Nagoya}
	For $S\subset\mathbb{\ol C}$ we define $E_f(S,k)=\cup_{a\in S}E_k(a;f)$, where k is a non-negative integer $a\in S$ or infinity. Clearly $E_f(S)=E_f(S,\infty)$ and $\overline{E}_f(S)=E_f(S,0)$. If $E_f(S,k)=E_g(S,k)$, we say that $f$ and $g$ share the set $S$ with weight $k$.
\end{defi}
\begin{defi} \label{d3}\cite{Lahiri-int.j.sci} For $a\in\mathbb{C}\cup\{\infty\}$ we denote by $N(r,a;f\mid=1)$ the counting function of simple $a$-points of $f$. For a positive integer $m$ we denote by $N(r,a;f\vline\leq m) (N(r,a;f\vline\geq m))$ the counting function of those $a$-points of $f$ whose multiplicities are not greater(less) than $m$ where each $a$-point is counted according to its multiplicity.
	
	$\ol N(r,a;f\mid\leq m) (\ol N(r,a;f\mid\geq m))$ are defined similarly, where in counting the $a$-points of $f$ we ignore the multiplicities.
	
	Also $N(r,a;f\mid <m), N(r,a;f\mid >m), \ol N(r,a;f\mid <m)\; and \;\ol N(r,a;f\mid >m)$ are defined analogously.  \end{defi}

\begin{defi}\label{d5}\cite{Ban-Ann.polon} Let $f$ and $g$ be two non-constant meromorphic functions such that $f$ and $g$ share $(a,0)$. Let $z_{0}$ be an $a$-point of $f$ with multiplicity $p$, an $a$-point of $g$ with multiplicity $q$. We denote by $\ol N_{L}(r,a;f)$ the reduced counting function of those $a$-points of $f$ and $g$ where $p>q$, by $N^{1)}_{E}(r,a;f)$ the counting function of those $a$-points of $f$ and $g$ where $p=q=1$, by $\ol N^{(2}_{E}(r,a;f)$ the reduced counting function of those $a$-points of $f$ and $g$ where $p=q\geq 2$. In the same way we can define $\ol N_{L}(r,a;g),\;  N^{1)}_{E}(r,a;g),\;  \ol N^{(2}_{E}(r,a;g).$ In a similar manner we can define $\ol N_{L}(r,a;f)$ and $\ol N_{L}(r,a;g)$ for $a\in\mathbb{C}\cup\{\infty\}$. \end{defi}
When $f$ and $g$ share $(a,m)$, $m\geq 1$, then $N^{1)}_{E}(r,a;f)=N(r,a;f\mid=1)$.
\begin{defi} \label{d7}\cite{Lahiri-Nagoya,Lahiri-comp.var} Let $f$, $g$ share a value $a$ IM. We denote by $\ol N_{*}(r,a;f,g)$ the reduced counting function of those $a$-points of $f$ whose multiplicities differ from the multiplicities of the corresponding $a$-points of $g$.
	
	Clearly $\ol N_{*}(r,a;f,g) = \ol N_{*}(r,a;g,f)=\ol N_{L}(r,a;f)+\ol N_{L}(r,a;g)$
\end{defi}
\begin{defi}\cite{Fujimoto-Am.J.Math}
	Let $P(z)$ be a polynomial such that $P^{'}(z)$ has mutually $k$ distinct  zeros given by $d_{1}, d_{2}, \ldots, d_{k}$ with multiplicities $q_{1}, q_{2}, \ldots, q_{k}$ respectively. Then $P(z)$ is said to be a critically injective polynomial if $P(d_i)\not =P(d_j)$ for $i\not=j$, where $i,j\in \{1,2,\cdot\cdot\cdot,k\}$ .
\end{defi}
\par From the definition it is obvious that $P(z)$ is injective on the set of distinct zeros of $P'(z)$ which are known as the critical points of $P(z)$. Thus a critically injective polynomial has at-most one multiple zero. We first invoke the following polynomial used in \cite{Ban-Mallick-bohemica}. 
\par We denote by $P(z)=z^{n}+az^{n-m}+bz^{n-2m}+c$ and\\ $\beta_i=-\left( c_{i}^{n}+ac_{i}^{n-m}+bc_{i}^{n-2m}\right)$, where $n, m\in\mathbb{N}$ and $a,b,c\in\mathbb{C^{*}}$ be such that $a^{2}\neq4b$, $\gcd(m,n)=1$, $n>2m$  and $c_i$ be the roots of the equation \be\label{e1} nz^{2m}+a(n-m)z^{m}+b(n-2m)=0,\ee  for $i=1,2,\ldots,2m$. Note that when $\frac{a^2}{4b}=\frac{n(n-2m)}{(n-m)^{2}}$, then (\ref{e1}) reduces to the equation $$n\left(z^{m}+\frac{a(n-m)}{2n}\right)^2-\frac{a^{2}(n-m)^{2}}{4n}+b(n-2m)=0;$$ i.e., \be\label{e2}n\left(z^{m}+\frac{a(n-m)}{2n}\right)^2=0.\ee  Hence in this case (\ref{e1}) has $m$ distinct roots $c_{i}$, $i=1,2,\ldots,m$ each being repeated twice. 

\par In view of the above discussion, we have following theorems which are the main results of the paper.
\begin{theo}\label{t1.1} Let $S=\left\{z:z^{n}+az^{n-m}+bz^{n-2m}+c=0\right\}$, $S'=\{0,c_1,c_2,\ldots,c_m\}$, where $n(\geq 2m+3)$, $\gcd(m,n)=1$, $\frac{a^2}{4b}=\frac{n(n-2m)}{(n-m)^{2}}$ and $a$, $b$, $c\in\mathbb{C^{*}}$ be such that $c\not =\beta_i,\frac{\beta_i\beta_j}{\beta_i+\beta_j}$. Let $f$ be a non constant meromorphic function with finitely many poles and $\mathcal{L}$ be a non constant $L$-function such that   $E_f(S,0)=E_{\mathcal{L}}(S,0)$, $E_f(S',\infty)=E_{\mathcal{L}}(S',\infty)$. Then for $n\geq\max\{2m+3,7\}$ we get $f = \mathcal{L}$.  \end{theo}
\begin{cor}Putting $a=-\frac{2n}{n-1}$, $b=\frac{n}{n-2}$, $c=\frac{2c'}{(n-1)(n-2)}$ and $m=1$ in {\it{Theorem \ref{t1.1}}} we have $S=\{z:z^{n}-\frac{2n}{n-1}z^{n-1}+\frac{n}{n-2}z^{n-2}+\frac{2c'}{(n-1)(n-2)} \;(c'\not=0,-1)\}$ and $S'=\{0,1\}$.
	Clearly if a nonconstant meromorphic function $f$ with finitely many poles and a non constant $L$-function $\mathcal{L}$, such that $E_f(S,0)=E_{\mathcal{L}}(S,0)$, $E_f(S',\infty)=E_{\mathcal{L}}(S',\infty)$ then for $n\geq 7$ we will get $f=\mathcal{L}$. Hence for $m=1$ we get a particular case of {\it{Theorem F}}.
\end{cor}
\begin{theo}\label{t1.2} Let $S$ and $S'$ be defined as in {\em Theorem \ref{t1.1}}.  Let $f$ be a non constant meromorphic function with finitely many poles and $\mathcal{L}$ be a non constant $L$-function such that $E_{f}(S,s)=E_{\mathcal{L}}(S,s)$, and $E_{f}(\{\alpha\},0)=E_{\mathcal{L}}(\{\alpha\},0)$ for some $\alpha\in S'$. For \\(I) $\alpha=0$ and  
	\\(i) $s\geq 2$, $n\geq 2m+2$ or
	\\(ii) $s=1$, $n\geq 2m+3$ or 
	\\(iii) $s=0$, $n\geq2m+5$; we have $f = \mathcal{L}$.\par Next suppose \\(II) $\alpha\not=0$. If 
	\\(i) $s\geq 1$ and $n\geq 2m+4$ or 
	\\(ii) $s=0$ and $n\geq 2m+7$; then  
	we have $f = \mathcal{L}$.
\end{theo}

\section{lemmas}
Next, we present some lemmas that will be needed in the sequel. Henceforth, we denote by $H$, $\Phi$ the following functions:
\bea\label{e2.1} H=\bigg(\frac{F''}{F'}-\frac{2F'}{F-1}\bigg)-\bigg(\frac{G''}{G'}-\frac{2G'}{G-1}\bigg)\;\eea
and
\bea\label{e2.2}\Phi=\frac{F'}{F-1}-\frac{G'}{G-1}.\eea
Let $f$ and $g$ be two non-constant meromorphic functions and for an integer $n\geq 2m+1$ 
\be\label{e2.3} F=\frac{f^{n-2m}(f^{2m}+af^{m}+b)}{-c},\;\;\;  G=\frac{g^{n-2m}(g^{2m}+ag^{m}+b)}{-c}. \ee
\begin{lem}\label{l2.1}\cite{Yi-kodai(1999)}
	Let $F$ and $G$ share $1$ IM and $H\not\equiv 0$. Then,
	$$N^{1)}_{E}(r,1;F)\leq N(r,\infty;H)+S(r,F)+S(r,G).$$
\end{lem}
\begin{lem}\label{l2.2}\cite{Mokhon'ko-lem}
	Let $P(f)=\sum_{k=0}^{n} a_{k}f^{k}/ \sum_{j=0}^{m}b_{j}f^{j},$ be an irreducible polynomial in $f$, with constants coefficient $ \{a_{k}\}$ and $\{b_{j}\}$ where $a_{n}\not=0$ and $b_{m}\not=0$. Then $$T(r,P(f))=dT(r,f)+S(r,f),$$ where $d=\max\{m,n\}.$
\end{lem}
\begin{lem}\label{l2.3}\cite{Ban_Tamkang}
If $F$ and $G$ share $(1,s)$, $0\leq s<\infty$, then $$\ol N(r,1;F)+\ol N(r,1;G)+\left(s-\frac{1}{2}\right)\ol N_{*}(r,1;F,G)-N^{1)}_{E}(r,1;F) \leq \frac{1}{2}\big( N(r,1;F)+ N(r,1;G)\big).$$
\end{lem}
\begin{lem}\label{l2.4}
 Let $F$, $G$ be given by (\ref{e2.3}) and $E_{f}(S,s)=E_{g}(S,s)$  where $S$ is given as in {\em Theorem \ref{t1.1}} and $H\not\equiv 0$. Then we have \beas \; N(r,\infty;H)&\leq& \ol N(r,0;f)+\ol N\left(r,0;f^{m}+\frac{a(n-m)}{2n}\right)+\ol N(r,0;g)+\ol N\left(r,0;g^{m}+\frac{a(n-m)}{2n}\right)\\& &+\ol N(r,\infty;f)+\ol N(r,\infty;g)+\ol N_{*}(r,1;F,G)+\ol N_{0}(r,0;f^{'})+\ol N_{0}(r,0;g^{'})\\&&+S(r,f)+S(r,g),\eeas \\
where $\ol N_{0}(r,0;f^{'})$ is the reduced counting function of those zeros of $f^{'}$ which are not the zeros of $f(nf^{2m}+(n-m)af^{m}+b(n-2m))(F-1)$ and $\ol N_{0}(r,0;g^{'})$ is similarly defined.
\end{lem}
\begin{proof}Since $E_{f}(S,s)=E_{g}(S,s)$, clearly $F$ and $G$ share $(1,s)$.
	\par Again from (\ref{e2.3}) and from the condition $\frac{a^{2}}{4b}=\frac{n(n-2m)}{(n-m)^{2}}$ mentioned in { Theorem \ref{t1.1}} we get that \beas\label{e2.4}F^{'}=\frac{f^{n-2m-1}(nf^{2m}+(n-m)af^{m}+b(n-2m))f'}{-c}=\frac{nf^{n-2m-1}(f^{m}+\frac{a(n-m)}{2n})^2f^{'}}{-c},\eeas
	\beas\label{e2.4ext}G^{'}=\frac{g^{n-2m-1}(ng^{2m}+a(n-m)g^{m}+b(n-2m))g'}{-c}=\frac{ng^{n-2m-1}(g^{m}+\frac{a(n-m)}{2n})^2g^{'}}{-c}.\eeas
	 then $$\ol N(r,0;nf^{2m}+a(n-m)f^{m}+b(n-2m))=\ol N(r,0;f^{m}+\frac{a(n-m)}{2n}),$$ Similar result holds for $g$. Then clearly from the definition of $H$ we have \beas\; N(r,H)&\leq& \ol N(r,0;f)+\ol N\left(r,0;f^{m}+\frac{a(n-m)}{2n}\right)+\ol N(r,0;g)+\ol N\left(r,0;g^{m}+\frac{a(n-m)}{2n}\right)\\& &+\ol N(r,\infty;f)+\ol N(r,\infty;g)+\ol N_{*}(r,1;F,G)+\ol N_{0}(r,0;f^{'})+\ol N_{0}(r,0;g^{'})\\&&+S(r,f)+S(r,g).\eeas Hence the proof is complete.
\end{proof}
\begin{lem}\label{l2.5}
	Let $F$, $G$ be given by (\ref{e2.3}) and $E_{f}(S,0)=E_{g}(S,0)$, $E_{f}(S',\infty)=E_{g}(S', \infty)$, where $S$, $S'$ be given as in {\em  Theorem \ref{t1.1}}. Suppose $H\not\equiv 0$. Then for $\frac{a^2}{4b}=\frac{n(n-2m)}{(n-m)^{2}}$, we have \beas \; N(r,\infty;H)&\leq& 
	\chi_{n}\left(\ol N(r,0;f)+\ol N\left(r,0;f^{m}+\frac{a(n-m)}{2n}\right)\right)+\ol N(r,\infty;f)\\& &+\ol N(r,\infty;g)+\ol N_{*}(r,1;F,G)+\ol N_{0}(r,0;f^{'})+\ol N_{0}(r,0;g^{'}),\eeas \\
	where $\ol N_{0}(r,0;f^{'})$ is the reduced counting function of those zeros of $f^{'}$ which are not the zeros of $f(nf^{2m}+(n-m)af^{m}+b(n-2m))(F-1)$, $\ol N_{0}(r,0;g^{'})$ is similarly defined and $\chi_{n}=1$ when $n\not=2m+3$ and $\chi_{n}=0$ when $n=2m+3$.
\end{lem}
\begin{proof}
	We omit this proof since it can be easily obtained from the proof of {\it{Lemma 2.2}} \cite{Ban-Mallick-bohemica}.
\end{proof}
\begin{lem}\label{l2.6}
	Let $S$, $S'$ be defined as in {\em Theorem \ref{t1.1}} and $F$, $G$ be given by (\ref{e2.3}). Suppose for two non-constant meromorphic functions $f$ and $g$, $E_{f}(S,0)=E_{g}(S,0)$, $E_{f}(S',\infty)=E_{g}(S',\infty)$,  and $\Phi\not\equiv 0$. Then
	 for $\frac{a^2}{4b}=\frac{n(n-2m)}{(n-m)^{2}}$ with $n\geq 2m+3$, we have
	\beas  && \ol N\left(r,0;f\right)+\ol N\left(r,0;f^{m}+\frac{a(n-m)}{2n}\right)\nonumber\\&\leq&\frac{1}{2}\left(\ol N_{*}(r,1;F,G)+\ol N(r,\infty;f)+\ol N(r,\infty;g)\right)+S(r,f)+S(r,g).\eeas 
\end{lem}
\begin{proof}
	We omit this proof since it can be easily obtained from the proof of {\it{Lemma 2.5}} \cite{Ban-Mallick-bohemica}.
\end{proof}
\begin{lem}\label{l2.7}\cite{Ban-Mallick-cmft}
	Let $\phi(z)=a^2(z^{n-m}-A)^2-4b(z^{n-2m}-A)(z^n-A)$, where $A,a,b\in\mathbb{C^{*}}$, $\frac{a^2}{4b}=\frac{n(n-2m)}{(n-m)^{2}}$, $\gcd(m,n)=1$, $n>2m$. If $\omega^l$ is the m-th root of unity for $l=0,1,\ldots,m-1$, then  \\i) $\phi(z)$ has no multiple zero, when $A\neq\omega^l$. \\ii) $\phi(z)$ has exactly one multiple zero, when $A=\omega^l$ and that is of multiplicity 4.\\ In particular, when $A=1$, then the multiple zero is 1. 
\end{lem}
\begin{lem}\label{l2.8}\cite{Ban-Mallick-cmft}
	Let $P(z)=z^{n}+az^{n-m}+bz^{n-2m}+c$, where $a, b\in\mathbb{C^{*}}$.  
	Then the followings hold.\\
	i) $\beta_i$'s are non-zero if $a^2\neq4b$. \\
	ii) $P(z)$ is critically injective polynomial if $\frac{a^2}{4b}=\frac{n(n-2m)}{(n-m)^2}$.
\end{lem}
\begin{lem}\label{l2.9}
	Let $F$, $G$ be given by (\ref{e2.3}),  $E_{f}(S,s)=E_{g}(S,s)$, where $S$ is defined as in {\em {Theorem \ref{t1.1}}}. Then $$ \ol N_{L}(r,1;F)\leq \frac{1}{s+1}\big(\ol N(r,0;f)+\ol N(r,\infty;f)\big)+S(r,f).$$ Similar inequality holds for $G$.
\end{lem}
\begin{proof}
Since $E_f(S,s)=E_g(S,s)$, clearly $F$ and $G$ share $(1,s)$. From the choice of $c$, it is clear that the polynomial $P(z)=:z^{n}+az^{n-m}+bz^{n-2m}+c$ has no multiple zero, so we have \beas \ol N_{L}(r,1;F)&\leq& \ol N(r,1;F\mid\geq s+2)\\&\leq& \ol N(r,0;F'\mid\geq s+1;F=1)\\&\leq&\frac{1}{s+1}N(r,0;F'\mid\geq s+1;F=1)\\&\leq&\frac{1}{s+1}\big(N(r,0;f'\mid f\not=0)-N_{o}(r,0;f') \big)\\&\leq&\frac{1}{s+1}\big( N(r,0;\frac{f'}{f})-N_{o}(r,0;f') \big)\\&\leq&\frac{1}{s+1}\big(\ol N(r,\infty;f)+\ol N(r,0;f)-N_{o}(r,0;f') \big)+S(r,f)\\&\leq&\frac{1}{s+1}\big(\ol N(r,0;f)+\ol N(r,\infty;f)-N_{o}(r,0;f') \big)+S(r,f).\eeas  Here $N_{o}(r,0;f')=N(r,0;f'\mid f\not=0,\alpha_{1},\alpha_{2},\ldots,\alpha_{n})$, where $\alpha_{1},\alpha_{2},\ldots,\alpha_{n}$ are zeros of the polynomial $P(z)$.
\end{proof}
\begin{lem}\label{l2.10}
		Let $F$, $G$ be given by (\ref{e2.3}) and $\Phi\not=0$. Also let $E_{f}(S,s)=E_{g}(S,s)$, where $S$ is defined as in {\em {Theorem \ref{t1.1}}}, and $f$ and $g$ share $(0,0)$ then,
		\beas \ol N(r,0;f)=\ol N(r,0;g)&\leq&\frac{1}{n-2m-1}\left(\ol N_L(r,1;F)+\ol N_L(r,1;G)+\ol N(r,\infty;F)+\ol N(r,\infty;G)\right)\\&&+S(r,F)+S(r,G).\eeas
	\end{lem}
\begin{proof}
	Since $f$, $g$ share $(0,0)$, it follows that \beas\ol N(r,0;f)=\ol N(r,0;g) &\leq& \frac{1}{n-2m-1}N(r,0;\Phi)\\&\leq& \frac{1}{n-2m-1}T(r,\Phi)+O(1)\\&\leq& \frac{1}{n-2m-1}N(r,\infty;\Phi)+S(r,F)+S(r,G)\\&\leq& \frac{1}{n-2m-1}\big(\ol N_{*}(r,1;F,G)+\ol N(r,\infty;F)+\ol N(r,\infty;G)\big)\\&&+S(r,F)+S(r,G)
	.\eeas
\end{proof}
\begin{lem}\label{l2.11}
	Let $f$ be a meromorphic function having finitely many poles in $\mathbb{C}$ and $S$  be defined as in {\it{Theorem \ref{t1.1} }}. If $f$ and a non constant $L$-function $\mathcal{L}$ share the set $S\;$ IM, then $\rho(f) =\rho(\mathcal{L}) = 1.$
\end{lem}
\begin{proof}
	Adopting the same procedure as done in Theorem 5, \{p. 6, \cite{Yan-Li-Yi_Lith}\} we can easily obtain $\rho(f)=\rho(\mathcal{L})=1$.
\end{proof}
\begin{lem}\label{l2.12}\cite{Lin-Lin-filomat}
	If $\mathcal{L}$ is a non-constant $L$-function, then there is no generalized Picard	exceptional value of $\mathcal{L}$ in the complex plane.
\end{lem}
\section{Proof of the theorems}

\begin{proof}[Proof of Theorem \ref{t1.1}]Let us consider \beas F=\frac{f^{n-2m}(f^{2m}+af^{m}+b)}{-c},\;\;\;  G=\frac{\mathcal{L}^{n-2m}(\mathcal{L}^{2m}+a\mathcal{L}^{m}+b)}{-c}. \eeas
Clearly $F$ and $G$ share $(1,0)$. Since $f$ has finitely many poles and $\mathcal{L}$ has at most one pole then $\ol N(r,\infty;f)=\ol N(r,\infty;\mathcal{L})=O(\log r)$. Also from {\it{Lemma \ref{l2.11}}} we have $\rho(f)=\rho(\mathcal{L})=1$. Therefore it is obvious that, $S(r,f)=S(r,\mathcal{L})=O(\log r)$. 	
	\par Now from {\it{Lemmas \ref{l2.1}}}, {\it\ref{l2.2}} {\it{\ref{l2.5}}}, {\it{\ref{l2.6}}} and putting $s=0$ in {\it{Lemma \ref{l2.3}}} and by the second fundamental theorem we have \beas\;\;\;\;(n+m)(T(r,f)+T(r,\mathcal{L}))&\leq& \ol N(r,1;F)+\ol N(r,1;G)+\sum_{i=0}^{m}\ol N(r,c_i;f)+\sum_{i=0}^{m}\ol N(r,c_i;\mathcal{L})\\\nonumber&&+\ol N(r,0;f)+\ol N(r,0;\mathcal{L})+\ol N(r,\infty;f)+\ol N(r,\infty;g)\\\nonumber&&-N_0(r,0;f')-N_0(r,0;\mathcal{L}')+S(r,f)+S(r,\mathcal{L}).\nonumber\eeas i.e., \bea \label{e3.1}\frac{n}{2}\left(T(r,f)+T(r,\mathcal{L})\right)	&\leq& \ol N(r,0;f)+\ol N(r,0;\mathcal{L})+\left(\frac{3}{2}+\frac{\chi_n}{2}\right)(\ol N_{L}(r,1;F)+\ol N_{L}(r,1;G))\\\nonumber&&+O(\log r)\\\nonumber&\leq& T(r)+\left(\frac{3}{2}+\frac{1}{2}\right)(\ol N_{L}(r,1;F)+\ol N_{L}(r,1;G))+O(\log r), \eea 	where $T(r)=T(r,f)+T(r,\mathcal{L})$.
	\par Clearly when 
	$n\geq 7$ in view of {\it{Lemma \ref{l2.9}}}, from (\ref{e3.1}) we get a contradiction.
	\par Therefore $H\equiv 0$ and so integrating both sides we get, \bea\label{e3.2}\frac{1}{G-1}=\frac{A}{F-1}+B,\eea where $A\not=0$, $B$ are two constants. From {\it{Lemma \ref{l2.2}}} and (\ref{e3.2}) we have, \bea\label{e3.3} T(r,\mathcal{L})=T(r,f)+O(1).\eea

We omit the rest of the proof of this theorem as it can be carried out in the line of proof of {\it{Theorem 1.1}} for $H\equiv0$ \cite{Ban-Mallick-bohemica}.	 
  \end{proof}
\begin{proof}[Proof of Theorem 1.2]
	Let $F$ and $G$ be given as in the proof of {\it{Theorem \ref{t1.1}}}. Since $E_{f}(S,s)=E_{g}(S,s)$ then clearly $F$ and $G$ share $(1,s)$. Also it is given that $E_{f}(\{\alpha\},0)=E_{\mathcal{L}}(\{\alpha\},0)$ where $\alpha\in S'$. Next we consider the following cases.
	\\{\bf\underline{Case-I.}}
	Let us take $\alpha=0$. Considering $H\not\equiv 0$ and using the same argument as in {\it{Lemma \ref{l2.4}}} we get \beas  N(r,\infty;H)&\leq& \ol N_*(r,0;f,\mathcal{L})+\ol N\left(r,0;f^{m}+\frac{a(n-m)}{2n}\right)+\ol N\left(r,0;\mathcal{L}^{m}+\frac{a(n-m)}{2n}\right)\\& &+\ol N(r,\infty;f)+\ol N(r,\infty;\mathcal{L})+\ol N_{*}(r,1;F,G)+\ol N_{0}(r,0;f^{'})+\ol N_{0}(r,0;\mathcal{L}^{'})\\&&+S(r,f)+S(r,\mathcal{L}).\eeas  Now proceeding same as in (\ref{e3.1})
	  we have \bea\label{e3.4} \frac{n}{2}T(r)\leq mT(r)+3\ol N(r,0;f)+\left(\frac{3}{2}-s\right)\ol N_*(r,1;F,G)+O(\log r). \eea
	  Next in view of {\it Definition 1.6}, using {\it{Lemma \ref{l2.10}}} in (\ref{e3.4}) we get \bea\label{e3.5}\frac{n}{2}T(r)\leq mT(r)+\left(\frac{3}{2}-s+\frac{3}{n-2m-1}\right)\ol N_*(r,1;F,G)+O(\log r).\eea
	  
	  \par Clearly when 
	  \beas&(i)& s\geq2,\; n\geq 2m+2 \;\;or\;\; when\\&(ii)&  s=1,\;\; n\geq 2m+3\;\; or\;\; when \\&(iii)&  s=0,\;\; n\geq 2m+5,\eeas using {\it{Lemma \ref{l2.9}}}, from (\ref{e3.5}) we get a contradiction.
	  \par Therefore $H\equiv 0$. Integrating both sides we get (\ref{e3.2}) and so from {\it{Lemma \ref{l2.2}}} we again have (\ref{e3.3}).
	 \\ {\bf\underline{Case-I-1.}} Suppose $B\not=0$. Then from (\ref{e3.2}) we get \be\label{e3.7} G-1\equiv \frac{F-1}{BF+A-B}.\ee 
	  {\bf\underline{Subcase-I-1.1}}
	  If $A-B\not =0$, then noting that $\frac{B-A}{B}\not =0,1,\infty$; from (\ref{e3.1}) we get 
	  $$\ol N(r,\frac{B-A}{B};F)=\ol N(r,\infty;G).$$
	  Therefore in view of {\it Lemma \ref{l2.3}} and (\ref{e3.3}) the second fundamental theorem yields
	  \beas nT(r,f)&\leq& \ol N(r,0;F)+\ol N(r,\infty;F)+\ol N(r,\frac{B-A}{B};
	  F)+S(r,F)\\&\leq& (2m+1)T(r,f)+\ol N(r,\infty;f)+\ol N(r,\infty;\mathcal{L})+S(r,f)\\&\leq&(2m+1)T(r,f)+O(\log r),\eeas which is a contradiction for $n\geq2m+2$.
	  \\{\bf\underline{Subcase-I-1.2.}} If $A-B=0$, then from (\ref{e3.7}) we have \be\label{e3.8}G-1=\frac{F-1}{BF}.\ee  (\ref{e3.8}) implies that $0$'s of $f$ and (${f}^{2m}+a{f}^{m}+b$) contributes to the poles of $G$. Since $\frac{a^{2}}{4b}=\frac{n(n-2m)}{(n-m)^{2}}$; i.e., $a^{2}\not=4b$, it follows that all the zeros of $z^{2m}+az^{m}+b$ are simple. Since $\ol N(r,\infty;G)=\ol N(r,\infty;\mathcal{L})$, $\mathcal{L}$ has at most one pole at $z=1$ and $m\geq 2$, we arrive at a contradiction. When $m=1$, let $\eta_i\;(i=1,2)$ be the zeros of $z^{2}+az+b$ and so the  $\{0,\eta_1,\eta_2\}$ points of $f$ will be the poles of $\mathcal{L}$. First using the second fundamental theorem, it is easy to verify that among these $\{0,\eta_1,\eta_2\}$ points, $f$ can not have two exceptional values, so $f$ may have only one exceptional value which implies $\mathcal{L}$ has more than one pole. Hence we arrive at a contradiction again. 
	  \\{\bf\underline{Case-I-2.}}
	  Suppose $B=0$. Then from (\ref{e3.2}) we get that $$F-1=A(G-1);$$ i.e.,  \be\label{e3.9}f^{n}+af^{n-m}+bf^{n-2m}\equiv A\left( \mathcal{L}^{n}+a\mathcal{L}^{n-m}+b\mathcal{L}^{n-2m}+c\frac{A-1}{A}\right)\ee 
	  and 
	  \be\label{e3.10}f^{n}+af^{n-m}+bf^{n-2m}+c(1-A)\equiv A\left( \mathcal{L}^{n}+a\mathcal{L}^{n-m}+b\mathcal{L}^{n-2m}\right). \ee
	 Since $f$ and $\mathcal{L}$ share $0$ IM and $\mathcal{L}$ has no exceptional value, from (\ref{e3.9}), (\ref{e3.10}) we get $A=1$.
	  \\{\bf\underline{Subcase-I-2.1.}} When $A=1$. Then we get $F\equiv G$; i.e.,\be\label{e3.11}\mathcal{L}^{n-2m}(\mathcal{L}^{2m}+a\mathcal{L}^m+b)\equiv f^{n-2m}(f^{2m}+af^m+b).\ee From (\ref{e3.11}) we have $f$, $\mathcal{L}$ share $0$ and $\infty$ CM. Then clearly $h=\frac{\mathcal{L}}{f}$ has no zero and no pole. Now putting $\mathcal{L}=fh$ in $F\equiv G$ we get
	  \be\label{e3.12}  f^{2m}(h^n-1)+af^m(h^{n-m}-1)+b(h^{n-2m}-1)=0.\ee 
	  \\{\bf\underline{Subcase-I-2.1.1.}}
	  If $h$ is constant, then as $f$ is non-constant so, $h^n=h^{n-m}=h^{n-2m}=1$. Since $\gcd(m,n)=1$, so $h=1$. 
	  Therefore $f\equiv \mathcal{L}$.
	  \\{\bf\underline{Subcase-I-2.1.2.}}	
	  If $h$ is non-constant, 
	  then from (\ref{e3.12}), in view of {\it{Lemma \ref{l2.7}}} we get 
	  \be\label{e3.13}\left( f^m
	  +\frac{a}{2}\frac{h^{n-m}-1}{h^n-1}\right)^2=\frac{\phi(h)}{4(h^n-1)^2}=\frac{a^2(h-1)^4(h-\nu_1)(h-\nu_2)\ldots(h-\nu_{2n-2m-4})}{4(h^n-1)^2},\ee where $\nu_i$'s are the distinct simple zeros of $\phi(h)$ and each   $\nu_i$ points of $h$ are of multiplicities at least $2$. Therefore by the second fundamental theorem we get 
	  \beas (2n-2m-4)T(r,h)&\leq& \sum\limits_{i=1}^{2n-2m-4}\ol N(r,\nu_i;h)+\ol N(r,0;h)+\ol N(r,\infty;h)+S(r,h)\nonumber\\&\leq&(n-m-2)T(r,h)+S(r,h),\eeas which is a contradiction for $n\geq2m+2$.
\\{\bf\underline{Case-II.}}	  Let us consider $\alpha(\not=0)\in S'$.\par Without loss of generality we may assume $\alpha=c_m$.
	Considering $H\not\equiv 0$ and by the same argument as in {\it{Lemma \ref{l2.4}}} we get \beas  N(r,\infty;H)&\leq& \ol N(r,0;f)+\ol N(r,0;\mathcal{L})+\sum_{i=0}^{m-1}\ol N(r,c_{i};f)+\sum_{i=0}^{m-1}\ol N(r,c_{i};\mathcal{L})+\ol N_*{(r,\alpha;f,\mathcal{L})}\\& &+\ol N(r,\infty;f)+\ol N(r,\infty;\mathcal{L})+\ol N_{*}(r,1;F,G)+\ol N_{0}(r,0;f^{'})+\ol N_{0}(r,0;\mathcal{L}^{'})\\&&+O(\log r).\eeas
	Now proceeding same as in (\ref{e3.1})
	we have \bea\label{e3.14}\;\;\;\;\;\; \frac{n}{2}T(r)&\leq&(m-1)T(r)+2(\ol N(r,0;f)+\ol N(r,0;\mathcal{L}))+\left(\frac{3}{2}-s\right)(\ol N_L(r,1;F)+\ol N_L(r,1;G))\\\nonumber&&+\ol N_*(r,\alpha;f,\mathcal{L})+O(\log r). \eea
	Now using {\it{Lemma \ref{l2.9}}} in (\ref{e3.14}) we get 
	\bea\label{e3.15}\;\; \frac{n}{2}T(r)&\leq&(m+1)T(r)+\frac{3-2s}{2(s+1)}\left(\ol N(r,0;f)+\ol N(r,0;\mathcal{L})\right)+\ol N(r,\alpha;f)+O(\log r). \eea
	\par Clearly when 
	\beas&(i)& s\geq1,\; n\geq 2m+4 \;\;or\;\; when\\&(iii)&  s=0,\;\; n\geq 2m+7; \eeas  from (\ref{e3.15}) we get a contradiction.
	\par Therefore $H\equiv 0$ and so integration again yields (\ref{e3.2}).
	\\{\bf\underline{Case-II-1.}} Suppose $B\not=0$. Then we again get (\ref{e3.7}).
	  So we have $$\ol N(r,\frac{B-A}{B};F)=\ol N(r,\infty;G),$$ where $A,\;A-B\not=0$.
	Now we consider the following sub cases: \\{\bf\underline{Subcase-II-1.1}} Suppose that $\frac{B-A}{B}=\frac{\beta_m}{c}$ where $\alpha=c_m$. Since $\frac{a^2}{4b}=\frac{n(n-2m)}{(n-m)^{2}}$, then we have  \be\label{e3.16}F^{'}=n\frac{f^{n-2m-1}\left(\prod\limits_{i=1}^{m}(f-c_i)\right)^2}{-c}f^{'}.\ee   Again $\frac{a^2}{4b}=\frac{n(n-2m)}{(n-m)^{2}}\neq1$ implies $a^2\neq4b$. Therefore by {\it Lemma \ref{l2.8}} we get  $\beta_m\neq0$ and $P(z)$ is critically injective. Since any critically injective polynomial can have at most one multiple zero, it follows that
	 	\be\label{e3.17} f^{n}+af^{n-m}+bf^{n-2m}+\beta_m=({f}-c_m)^3\prod\limits_{j=1}^{n-3}({f}-\xi_j),\ee where $\xi_j$'s are $(n-3)$ distinct zeros of $z^{n}+az^{n-m}+bz^{n-2m}+\beta_m$ such that $\xi_j\neq c_m,0$, $j=1,2,\ldots,n-3$. Then from (\ref{e3.7}) and (\ref{e3.17}) we have
	 		 \be\label{e3.18} B(G-1)\equiv \frac{-c(F-1)}{(f-c_m)^3\prod\limits_{j=1}^{n-3}(f-\xi_j)}.\ee Since $E_{f}(\{c_m\},0)=E_{g}(\{c_m\},0)$, so $c_m$ points of $f$ are not poles of $G$ and hence $c_m$ is an e.v.P. of $f$ and hence an e.v.P. of $\mathcal{L}$.  Therefore from {\it{Lemma \ref{l2.12}}} we arrive at a contradiction.
	\\{\bf\underline{Subcase-II-1.2}}
	Next suppose $\frac{B-A}{B}\not=\frac{\beta_m}{c}$. Since $A$ and $A-B$ are non zero then adopting the same procedure as done in {\bf{Subcase-I-1.1}} of this theorem again we can get a contradiction.
\\{\bf\underline{Subcase-II-1.3}} If $A-B=0$ then by {\bf{Subcase-I-1.2}} we arrived at a contradiction.
\\{\bf\underline{Subcase-II-2}} Assuming $B=0$ we get $$F-1=A(G-1)$$ and subsequently we can obtain (\ref{e3.9}), (\ref{e3.10}).
\\{\bf\underline{Subcase-II-2.1.}} Let $A\not=1$. Then as	  $c\neq 0$,  so $c\frac{(A-1)}{A}\neq0$ and at the same time by {\it Lemma \ref{l2.8}} we have $\beta_i\neq0$. Therefore we have the following subcases.
\\{\bf\underline{Subcase-II-2.1.1.}} 	  Suppose $c\frac{(A-1)}{A}=\beta_i$ for some $i\in\{1,2,\ldots,m\}$. Then we claim that $c(1-A)\neq\beta_j$ for any $j\in\{1,2,\ldots,m\}$. For if $c(1-A)=\beta_j$; then $A=\frac{c-\beta_j}{c}$ and since it is given that $c\frac{(A-1)}{A}=\beta_i$; i.e., $A=\frac{c}{c-\beta_i}$,
it follows that $\frac{c-\beta_j}{c}=\frac{c}{c-\beta_i}$; i.e., $c=\frac{\beta_i\beta_j}{\beta_i+\beta_j}$, a contradiction. Thus 
$z^{n}+az^{n-m}+bz^{n-2m}+c(1-A)=0$ has only simple roots say $\gamma_i$ for $i=1,2,\ldots,n$. So from (\ref{e3.10}), (\ref{e3.3}) and by using the second fundamental theorem we get \beas (n-1)T(r,f)&\leq& \displaystyle\sum_{i=1}^{n}\ol N(r,\gamma_i;f)+\ol N(r,\infty;f)+S(r,f)\\&\leq&(2m+1)T(r,\mathcal{L})+O(\log r),\eeas gives a contradiction for $n\geq 2m+3$.
 \\{\bf\underline{Subcase-II-2.1.2.}} Suppose $c\frac{(A-1)}{A}\not=\beta_i$ for all $i\in\{1,2,\ldots,m\}$.  So,   $z^{n}+az^{n-m}+bz^{n-2m}+c\frac{(A-1)}{A}=0$ has only simple roots say $\mu_i$ for $i=1,2,\ldots,n$.
Therefore from (\ref{e3.9}), (\ref{e3.3}) and by the second fundamental theorem we have \beas (n-1)T(r,\mathcal{L})&\leq& \displaystyle\sum_{i=1}^{n}\ol N(r,\mu_i;\mathcal{L})+\ol N(r,\infty;\mathcal{L})+S(r,\mathcal{L})\\&\leq&(2m+1)T(r,f)+O(\log r),\eeas  gives a contradiction for $n\geq 2m+3$.\\
{\bf\underline{Subcase-II-2.2.}} Suppose $A=1$. Then we get $F=G$ and hence we obtain (\ref{e3.11}). Putting $\mathcal{L}=fh$ in (\ref{e3.11}) we get (\ref{e3.12}). 
	\par Now proceeding the same way as done in {\bf{Subcase-I-2.1.1}}-{\bf{Case-I-2.1.2}}  of this theorem, we will get $f\equiv \mathcal{L}$, for $n\geq 2m+4$.
	
\end{proof}


\begin{thebibliography}{100}
	\bibitem{Ban-Ann.polon} A. Banerjee, Some uniqueness results on meromorphic functions sharing three sets, Ann. Polon. Math.,  $\mathbf{92(3)}$ (2007), 261-274.
	\bibitem{Ban_Tamkang} A. Banerjee,  Uniqueness of meromorphic functions sharing two sets with finite weight II, Tamkang J. Math., $\mathbf{41}$(2010), 379-392.
	\bibitem{Ban-Mallick-cmft} A. Banerjee and S Mallick,  On the characterizations of a new class of strong uniqueness polynomials generating unique range sets, Comput. Methods Funct. Theo.,  $\mathbf{17(1)}$ (2017), 19-45.
	\bibitem{Fujimoto-Am.J.Math} H. Fujimoto, On uniqueness of meromorphic functions sharing finite sets, Amer. J. Math., $\mathbf{122}$(2000), 1175-1203.
	\bibitem{Graun-Grahl-Steuding} R. Garunkštis, J. Grahl and J. Steuding, Uniqueness theorems for $L$-functions, Comment. Math. Univ. St. Pauli,
	$\mathbf{60}$(2011), 15-35. 
	\bibitem{G.F_Springer(1977)} F. Gross, Factorization of meromorphic functions and some open problems Complex Analysis (Proc. Conf. Univ. Kentucky, Lexington, Kentucky, 1976), Lecture Notes in Math. $\mathbf{599}$, Springer-Berlin(1977), 51-69.
	\bibitem{W.K.Hayman_64} W. K. Hayman, Meromorphic functions, Oxford Mathematical Monographs, Clarendon Press, Oxford (1964).
	\bibitem{Hu_LI_Can-16} P. C. Hu and B. Q. Li, A simple proof and strengthening of a uniqueness theorem for $L$-functions, Can. Math. Bull.,
		$\mathbf{59}$(2016), 119-122.
	\bibitem{Lahiri-int.j.sci} I. Lahiri, Value distribution of certain differential polynomials, Int. J. Math. Math. Sci.,  $\mathbf{28(2)}$(2001), 83-91.
	\bibitem{Lahiri-Nagoya} I. Lahiri, Weighted sharing and uniqueness of meromorphic functions, Nagoya Math. J.,  $\mathbf{161}$(2001), 193-206.
	\bibitem{Lahiri-comp.var} I. Lahiri, Weighted value sharing and uniqueness of meromorphic functions, Complex Var. Theory Appl., $\mathbf{46(3)}$(2001), 241-253.
	\bibitem{B.Q.LI-Proc.Am-10} B. Q. Li, A result on value distribution of $L$-functions, Proc. Am. Math. Soc., $\mathbf{138(6)}$(2010), 2071-2077.
	\bibitem{Lin-Lin-filomat}P. Lin and W. Lin,  Value distribution of $L$-functions concerning sharing sets, Filomat, $\mathbf{30}$(2016), 3795-3806.
	\bibitem{Li-Yi_Nachr} X. M. Li and H. X. Yi, Results on value distribution of $L$-functions, Math. Nachr., $\mathbf{286}$(2013), 1326-1336.
	\bibitem{Mokhon'ko-lem}A. Z. Mokhon'ko, On the Nevanlinna characteristics of some meromorphic functions, Theory of functions. Functional analysis and their applications, $\mathbf{14}$(1971), 83-87.
	\bibitem{Hal-Sahoo-cmft} P. Sahoo and S. Halder, Some results on L-functions related to sharing two finite sets, Comput. Methods Funct. Theo., $\mathbf {19}$(2019), 601-612.
	\bibitem{Ban-Mallick-bohemica} S. Mallick and A. Banerjee, Uniqueness and two shared set problems of meromorphic function in a different angle, math. Bohemica (2019), 1-27 (DOI: 10.21136/MB.2019.0063-18).
	\bibitem{selberg-92} A. Selberg, Old and new conjectures and results about a class of Dirichlet series, in: Proccedings of the Amalfi	Conference on Analytic Number Theory (Maiori, 1989), Univ. Salerno, Salerno (1992), 367-385.
	\bibitem{Steuding-Sprin-07} J. Steuding, Value Distribution of $L$-Functions, Lect. Notes Math., Vol. 1877, Springer, Berlin (2007).
	
	\bibitem{Yan-Li-Yi_Lith} Q. Q.  Yuan, X. M.  Li, and H. X. Yi, Value distribution of $L$-functions and uniqueness questions of F. Gross, Lithuanian Math. J., $\mathbf{58(2)}$(2018), 249-262.
	\bibitem{Yi-kodai(1999)} H. X. Yi, Meromorphic functions that share one or two values II, Kodai Math. J., $\mathbf{22}$(1999), 264-272.
\end{thebibliography}
\end{document}